\documentclass[11pt]{article}
\usepackage{amsmath, amssymb, amsthm}
\usepackage{geometry}
\usepackage{booktabs}
\theoremstyle{plain}
\newtheorem{theorem}{Theorem}[section]
\newtheorem{lemma}[theorem]{Lemma}

\theoremstyle{definition}

\theoremstyle{remark}

\newcommand{\eps}{\varepsilon}

\begin{document}
\begin{center}
	{\LARGE \textbf{On the Mean Value of $D_k(n)$ in Arithmetic Progressions}}\\[1cm]
	
	{\large Meselem KARRAS}\\[0.3cm]
	
	Faculty of Science and Technology, Tissemsilt University, Algeria.\\
	FIMA Laboratory, Khemis Miliana University, Algeria.\\
	Email: \texttt{m.karras@univ-tissemsilt.dz}
\end{center}

\bigskip

\noindent

\begin{abstract}
	Let  \(k \ge 2\) be a fixed integer. We define the multiplicative function \(D_k(n) = d_k(n)/d_k^*(n)\), such that \(d_k(n)\) is the Piltz divisor function and \(d_k^*(n) = k^{\omega(n)}\) is its unitary analogue, where $\omega(n)$ is the number of distinct prime divisors of $n$. We establish an asymptotic formula for the sum
	\[
	\sum_{\substack{n \le x \\ n \equiv a \pmod q}} D_k(n),
	\]
	where \(\gcd(a,q)=1\). This result is a generalization of the study presented in \cite{Derbal 2023}.
\end{abstract}
	
\section{Introduction}

The Dirichlet–Piltz divisor function $d_k$ is defined by $d_1=1$, and for a fixed integer $k\ge 2$ and $n\ge 1$
\[ d_k(n)=\sum_{n_1\dots n_k=n}{1}=\sum_{ml=n}d_{k-1}(m).
\]
The unitary analogue $d_k^*$ is defined as follows \[d_k^*(n)=k^{\omega(n)},\] where $\omega(n)$ is the number of distinct prime divisors of $n$. The arithmetic function
\begin{equation}
D_k(n)=\frac{d_k(n)}{d_k^*(n)},
\end{equation}\label{eq:1}
is multiplicative, and for any prime power $p^m$, 
\begin{equation*}
	D_k(p^m)=\frac{1}{k}\binom{k+m-1}{m}.
\end{equation*}

This function defined by \eqref{eq:1} was studied in \cite{Karras 2018}, and the authors gave an asymptotic formula for its summation function as follows:
\begin{equation}\label{eq:KarrasResult}
	\sum_{n\le x}D_k(n)=A_kx+O\!\bigl(x^{1/2}(\log x)^{k-2}\bigr),
\end{equation}
where
\begin{equation}\label{eq:Ak}
	A_k=\prod_{p}\Bigl(1-\frac1p\Bigr)\Bigl(1-\frac1k+\frac1k\Bigl(1-\frac1p\Bigr)^{-k}\Bigr).
\end{equation}
Some values of $A_k$ (rounded to $4$ decimal places):
\begin{center}
	\begin{tabular}{c|ccccccc}
		$k$       & 2      & 3      & 4      & 5      & 6       & 7       & 8       \\
		\hline
		$A_k$     & 1.4276 & 2.2239 & 3.7997 & 7.1067 & 14.4445 & 31.5962 & 73.6569
	\end{tabular}
\end{center}

In particular, in the case $k=2$, an asymptotic formula was obtained in \cite{Karras 2016} by
\[
\sum_{n\le x}\frac{d(n)}{d^*(n)} = A_2 x + O(x^{1/2}),
\]
with $$A_2 = \frac{\pi^2}{6}\prod_p(1-\frac{1}{2p^2}+\frac{1}{2p^3}).$$

On the other hand, in the same particular case, the authors in \cite{Derbal 2023} studied this summation function and obtained an asymptotic formula for over arithmetic progression, with this error $O(x\exp(-c\sqrt{\log x\log\log x}))$. Their proof is based on Perron's formula and Ikehara's theorem.
These studies of this function motivate us to investigate it in arithmetic progressions and to derive an asymptotic formula.

In this paper, we will study the following sum
\begin{equation}\label{eq:4}
S_k(x;q,a)=\sum_{\substack{n\le x\\ n\equiv a\pmod q}}D_k(n),
\end{equation}
where $a$ and $q$ are integers, $q\ge1$ and $(a,q)=1$.
Our object is to establish an asymptotic formula for this sum using elementary methods. 

\section{A useful auxiliary theorem}

	First, in this section, we present an important result on the sum of function $D_k(n)$ over integers $n\le x$ that are coprime to $q$.

\begin{theorem}\label{thm:1}
	For any $x \geq 2$ and any $\eps > 0$, we have
\begin{equation}\label{eq:5}
\sum_{\substack{n \leq x \\ (n,q)=1}} D_k(n) = G_k(q) \cdot x 
+ O_{q, \eps}\left(2^{\omega(q)}\cdot x^{1/2+\eps} \right),
\end{equation}
where $\omega(q)$ denotes the number of distinct prime factors of $q$ and $$G_k(q)=A_k\cdot\prod_{p \mid q} \frac{k}{k-1 + (1-1/p)^{-k}}.$$

In particular, if $q = O(x^\delta)$ for some fixed $\delta > 0$, then for any $\eps > 0$,
\begin{equation}\label{eq:6}
\sum_{\substack{n \leq x \\ (n,q)=1}} D_k(n) =  G_k(q) \cdot x + O_{\eps}\left(x^{1/2+\eps}\right).
\end{equation}
\end{theorem}

	The proof of Theorem \ref{thm:1} requires several lemmas.	We begin with the following convolution identity, established in \cite{Karras 2018}, 
	which plays a fundamental role in our analysis.
	
	\begin{lemma}[Lemma 4 in \cite{Karras 2018}]\label{lem:convolution_Dk}
		We have the convolution identity:
 $D_k * \mu = g_k$, where
	\begin{equation}\label{eq:7}
		g_k(n) = \frac{k-1}{k} \cdot s_2(n) \cdot D_{k-1}(n),
	\end{equation}
		and $s_2(n)$ denotes the characteristic function of 2-full integers (i.e., $s_2(n)=1$ if for every prime $p$ dividing $n$, we have $p^2 \mid n$, and $s_2(n)=0$ otherwise). Consequently, $D_k = g_k * 1$.
	\end{lemma}

	\begin{lemma}\label{lem:sum_via_convolution}
		For any $x \geq 1$ and $q \geq 1$, we have
	\begin{equation}\label{eq:12}
		\sum_{\substack{n \leq x \\ (n,q)=1}} D_k(n) = 
		\frac{\varphi(q)}{q} \, x \sum_{\substack{d \leq x \\ (d,q)=1}} \frac{g_k(d)}{d} 
		+ O\!\Biggl( 2^{\omega(q)} \sum_{\substack{d \leq x \\ (d,q)=1}} |g_k(d)| \Biggr)
	\end{equation}
	\end{lemma}
	
	\begin{proof}
		Since $D_k = g_k * 1$, we have $D_k(n) = \sum_{d \mid n} g_k(d)$. Therefore,
		\begin{align*}
			\sum_{\substack{n \leq x \\ (n,q)=1}} D_k(n) 
			&= \sum_{\substack{n \leq x \\ (n,q)=1}} \sum_{d \mid n} g_k(d) \\
			&= \sum_{d \leq x} g_k(d) \sum_{\substack{n \leq x \\ d \mid n \\ (n,q)=1}} 1.
		\end{align*}
		By writing $n = d m$, implies that the condition $d \mid n$ becomes $m \leq x/d$, and the condition $(n,q)=1$ is equivalent to $(d,q)=1$ and $(m,q)=1$. Then by the Legendre totient function $\varphi(.,.)$ (see. e.g,.\cite{Bord O} page $329$), we have
		
	\begin{align*}
		\sum_{\substack{n \leq x \\ (n,q)=1}} D_k(n) 
		&= \sum_{\substack{d \leq x \\ (d,q)=1}} g_k(d) \sum_{\substack{m \leq x/d \\ (m,q)=1}} 1 \\
		&= \sum_{\substack{d \leq x \\ (d,q)=1}} g_k(d) \Biggl( \frac{\varphi(q)}{q} \cdot \frac{x}{d} + O\!\Bigl(2^{\omega(q)}\Bigr) \Biggr).
	\end{align*}

	\end{proof}

	\begin{lemma}\label{lem:4}
		For $\Re(s) > 1/2$, let 
		\[
		G_{k,q}(s) = \sum_{\substack{n=1\\ (n,q)=1}}^{\infty} g_k(n) n^{-s}.
		\]
		Then, for any $\eps > 0$,
		\begin{equation}\label{eq: 13}		
		\sum_{n > x} \frac{|g_k(n)|}{n} \ll x^{-1/2+\eps}.
	\end{equation}
		Moreover,
		\begin{equation}\label{eq: 14}	
		\sum_{\substack{n \le x \\ (n,q)=1}} \frac{g_k(n)}{n} = G_{k,q}(1) + O\bigl(x^{-1/2+\eps}\bigr),
			\end{equation}
		where
$$	G_{k,q}(1) = A_k \cdot \prod_{p \mid q} \frac{k}{(1-1/p) \bigl(k-1 + (1-1/p)^{-k}\bigr)}.$$
	
	\end{lemma}	
	\begin{proof}
		First, the estimate \eqref{eq: 13} follows from the fact that $g_k(n) \neq 0$ only if $n$ is 2-full, and for such $n$, $|g_k(n)| \ll_k n^\epsilon$. On the other hand, we know that the number of  $2$‑full integers $\le t$ is $O(\sqrt{t})$ (see e.g. \cite{Ba Gr}). Since we have
		\[
		\sum_{n > x} \frac{|g_k(n)|}{n} \ll_k \sum_{\substack{n > x \\ n \text{ 2-full}}} \frac{n^\epsilon}{n} \ll x^{-1/2+\epsilon}.
		\]
		
		For equality \eqref{eq: 14}. Since $g_k$ is multiplicative and with the condition $(n,q)=1$, we have
		\[
		G_{k,q}(s) = \prod_{p \nmid q} \left(1 + \sum_{m \geq 2} \frac{g_k(p^m)}{p^{ms}}\right).
		\]
		Becomes $g_k(p)=0$, the sum starts at $m=2$. At $s=1$, from Lemma 4 of \cite{Karras 2018}, we know that
		\[
		\prod_{p} \left(1 + \sum_{m \geq 2} \frac{g_k(p^m)}{p^{m}}\right) = A_k.
		\]
		Therefore,
		\[
		G_{k,q}(1) = A_k  \prod_{p \mid q} \left(1 + \sum_{m \geq 2} \frac{g_k(p^m)}{p^{m}}\right)^{-1}.
		\]
By the definition of $g_k$ from \eqref{eq:7}, the generating series of $D_{k-1}$, and the identity
\[
\sum_{m=0}^\infty \binom{k-1+m}{m} z^m = (1-z)^{-k},
\]
we get
\[
1 + \sum_{m \ge 2} \frac{g_k(p^m)}{p^{m}} 
= (1-1/p) \left[ 1 - \frac{1}{k} + \frac{1}{k} (1-1/p)^{-k} \right].
\]

which completes the proof.
	\end{proof}

\begin{lemma}\label{lem:5}
	For any $x \ge 2$, we have
\begin{equation}\label{eq:8}
	\sum_{\substack{n \le x \\ (n,q)=1}} |g_k(n)| \ll x^{1/2} (\log x)^{k-2}.
\end{equation}
\end{lemma}

\begin{proof}
	From \eqref{eq:7} we have $|g_k(n)| \le \frac{k-1}{k} D_{k-1}(n)$ for every 2‑full $n$, and $g_k(n)=0$ otherwise. Consequently, for any $q\ge1$,
	\begin{equation}\label{eq:9}
		\sum_{\substack{n \le x \\ (n,q)=1}} |g_k(n)| \le \frac{k-1}{k} \sum_{\substack{n \le x \\ (n,q)=1 \\ n\text{ is 2‑full}}} D_{k-1}(n)  
		\le \frac{k-1}{k} \sum_{\substack{n \le x \\ n\text{ is 2‑full}}} D_{k-1}(n).
	\end{equation}
	
	For every $2$‑full integer $n$ can be written uniquely as $n = a^2 b^3$ with $\mu(b)^2 = 1$. 
	By Lemma~2 of \cite{Karras 2018} we have the sub‑multiplicative bound
	\[
	D_{k-1}(a^2 b^3) \le D_{k-1}(a^2) D_{k-1}(b^3) k^{\omega(b)}.
	\]
	Thus we have
	\[
	\sum_{\substack{n \le x \\ n\text{ 2‑full}}} D_{k-1}(n)
	\le \sum_{b\le x^{1/3}} \mu(b)^2 D_{k-1}(b^3) k^{\omega(b)} \sum_{a\le\sqrt{x/b^3}} D_{k-1}(a^2).
	\]
	
	Lemma~5 of \cite{Karras 2018} gives $\sum_{a\le y} D_{k-1}(a^2) \ll y (\log y)^{k-2}$, and
	Corollary~7 of the same paper, this result yields
	\[
	\sum_{b\le y} \mu(b)^2 D_{k-1}(b^3) k^{\omega(b)} b^{-3/2} \ll 1.
	\]
According to these latest estimates, we obtain
	\begin{equation}\label{eq:10}
	\sum_{\substack{n \le x \\ n\text{ 2‑full}}} D_{k-1}(n)
	\ll (\log x)^{k-2} \sum_{b\le x^{1/3}} \mu(b)^2 D_{k-1}(b^3) k^{\omega(b)} \sqrt{\frac{x}{b^3}}
	\ll x^{1/2} (\log x)^{k-2}.
\end{equation}
	
	The desired bound \eqref{eq:8} follows by combining \eqref{eq:9} with \eqref{eq:10}.
	
\end{proof}

\begin{proof}[\textbf{Proof of Theorem~\ref{thm:1}}]
		First, by Lemma \ref{lem:sum_via_convolution}, we have
		\begin{align*}
			\sum_{\substack{n \leq x \\ (n,q)=1}} D_k(n) 
			&= \frac{\varphi(q)}{q} x \sum_{\substack{d \leq x \\ (d,q)=1}} \frac{g_k(d)}{d} + O\left(2^{\omega(q)} \sum_{\substack{d \leq x \\ (d,q)=1}} |g_k(d)|\right).
		\end{align*}
		
	According to estimates \eqref{eq: 14} and \eqref{eq:8}, we arrive at this estimate
	
		\begin{align*}
			\sum_{\substack{n \leq x \\ (n,q)=1}} D_k(n) 
			&= \frac{\varphi(q)}{q} x \left( G_{k,q}(1) + O(x^{-1/2+\eps}) \right) + O\left(2^{\omega(q)} x^{1/2} (\log x)^{k-2}\right) \\
			&= \frac{\varphi(q)}{q} G_{k,q}(1) x +  O\left(2^{\omega(q)} x^{1/2+\eps} \right).
		\end{align*}
		Since 
		\[
		\frac{\varphi(q)}{q} = \prod_{p \mid q} \Bigl(1 - \frac{1}{p}\Bigr)
		\quad \text{and} \quad 
		G_{k,q}(1) = A_k  \prod_{p \mid q} \frac{k}{((1-1/p) (k-1 + (1-1/p)^{-k})},
		\]
		we obtain the stated result.
	\end{proof}
	
\bigskip \noindent 
\section{Main result}

We now are in a position to prove the main result of this section, giving an asymptotic formula for $S_k(x;q,a)$.

\begin{theorem}\label{thm:2}
	Let $k \ge 2$ be a fixed integer, and let $a,q$ be integers such that 
$1 \le a \le q$ and $(a,q)=1$. Assume that $q = O(x^\delta)$ for some fixed $\delta > 0$.
Then for any $x \ge 2$ and any $\eps > 0$, we have
	\begin{equation}\label{eq:17}
	S_k(x; q, a) = \frac{G_k(q)}{\varphi(q)} \cdot{x} + O_{k, q}\Bigl( \sqrt{q} \,  x^{1/2+\epsilon} \Bigr),
	\end{equation}
	where $G_k(q)$ is defined as in Theorem~\ref{thm:1}.

\end{theorem}

The proof of Theorem \ref{thm:2} is based on the following lemmas. We first recall the orthogonality relation for Dirichlet characters, it has allowed us to separate the arithmetic progressions sums.
	
	Let $q \geq 1$ be an integer. We denote by $\widehat{G}(q)$ the set of Dirichlet characters modulo $q$. For $(a,q)=1$, the orthogonality relation for characters states ( see e.g.\cite{Tom A}):
	\begin{equation}\label{eq:15}
	\frac{1}{\varphi(q)} \sum_{\chi \in \widehat{G}(q)} \overline{\chi}(a) \chi(n) = 
	\begin{cases}
		1 & \text{if } n \equiv a \pmod{q},\\
		0 & \text{otherwise}.
	\end{cases}
\end{equation}
	
	\begin{lemma}\label{lem:filtration}
		With the above notation,
		\begin{equation}\label{eq:16}
		S_k(x; q, a) = \frac{1}{\varphi(q)} \sum_{\chi \in \widehat{G}(q)} \overline{\chi}(a) \sum_{n \leq x} \chi(n) D_k(n).
	\end{equation}
	\end{lemma}
	
	\begin{proof}
		Using \eqref{eq:15} in the definition of $S_k(x; q, a)$ and interchange the sums, we obtain the following identity, recorded as \eqref{eq:16} .
	\end{proof}

\begin{lemma}\label{lem:nonprincipal}
	Let $\chi \neq \chi_0$ be a non-principal Dirichlet character modulo $q$ ($q \ge 1$). 
	For $x \ge 2$, we have
	\[
	S_{\chi}(x) := \sum_{n \le x} \chi(n) D_k(n) \ll_k \sqrt{q} \, (\log q) \, x^{1/2} (\log x)^{k-2}.
	\]
	
\end{lemma}

\begin{proof}
	Using the convolution identity $D_k = g_k * 1$ (Lemma~\ref{lem:convolution_Dk}), we can write
	\[
	S_{\chi}(x) = \sum_{d \le x} g_k(d) \, \chi(d) \sum_{m \le x/d} \chi(m).
	\]
	For a non-principal character $\chi$, the Pólya--Vinogradov inequality implies that
	\[
	\Bigl| \sum_{m \le y} \chi(m) \Bigr| \ll \sqrt{q} \, \log q \qquad (y \ge 1).
	\]
	Since $|\chi(d)| \le 1$ for all $d$, it follows that
	\[
	|S_{\chi}(x)| \le \sum_{d \le x} |g_k(d)| \, \Bigl| \sum_{m \le x/d} \chi(m) \Bigr| \ll \sqrt{q} \, \log q \sum_{d \le x} |g_k(d)|.
	\]
	
	Next, by Lemma~\ref{lem:5} (taking $q=1$), we get
	\[
	\sum_{d \le x} |g_k(d)| \ll_k x^{1/2} (\log x)^{k-2}.
	\]
	Combining these estimates, we obtain
\[
S_{\chi}(x) \ll_k \sqrt{q} \, (\log q) \, x^{1/2} (\log x)^{k-2},
\]

	as required.
\end{proof}

\begin{proof}[\textbf{Proof of Theorem~\ref{thm:2}}]
	By Lemma~\ref{lem:filtration}, we can write
	\[
	S_k(x; q, a) = \frac{1}{\varphi(q)} \sum_{\chi \in \widehat{G}(q)} \overline{\chi}(a) \, S_\chi(x),
	\quad\text{where } S_\chi(x) = \sum_{n \le x} \chi(n) D_k(n).
	\]
	
	We separate the contribution of the principal character $\chi_0$ from that of the non-principal characters. For the principal character, we have $\chi_0(n) = \mathbf{1}_{(n,q)=1}$, and
	\[
	S_{\chi_0}(x) = \sum_{\substack{n \le x \\ (n,q)=1}} D_k(n). 
	\]
	If we assume $q = O(x^\delta)$ for some fixed $\delta>0$, then, $2^{\omega(q)} \ll_\eps x^\eps$ for any $\eps>0$, and by  Theorem~\ref{thm:1} it follows that
	\begin{equation} \label{eq:17}
	S_{\chi_0}(x) = G_k(q) x + O_{q, \eps}\Bigl( x^{1/2+\epsilon} \Bigr).
		\end{equation}
	
	\medskip\noindent
	
	For each non-principal character $\chi \neq \chi_0$, Lemma~\ref{lem:nonprincipal} gives
\[
S_{\chi}(x) \ll_{k, q} \sqrt{q} \, (\log q) \, x^{1/2} (\log x)^{k-2}.
\]

Under the same hypothesis $q = O(x^\delta)$ we have $(\log q)(\log x)^{k-2} \ll_{q,\eps} x^\eps$, and therefore
\[
S_{\chi}(x) \ll_{q,\eps} \sqrt{q} \, x^{1/2+\eps}.
\]
	
	Since there are exactly $\varphi(q)-1$ non-principal characters, the total contribution of all non-principal characters is of order
	\[
	\frac{1}{\varphi(q)} \sum_{\chi \neq \chi_0} |S_\chi(x)|  \ll_{q, \eps}\Bigl(\sqrt{q} \, x^{1/2+\epsilon} \Bigr).
	\]
	
	Combining the two sums of the principal and non-principal characters, we obtain
	\[
S_k(x; q, a) = \frac{G_k(q)}{\varphi(q)} x + O_{q, \eps}\Bigl( \sqrt{q} \,  x^{1/2+\epsilon} \Bigr),
\]
	as stated.
\end{proof}

\begin{table}[h]
	\centering
	\text{Some Values of $G_k(q)/\varphi(q)$  (4 decimals) from PARI/GP}
	\label{tab:Gk_phi_pari}
	\begin{tabular}{c|cccccccc}
		\toprule
		$k$ & $q=2$ & $q=3$ & $q=4$ & $q=5$ & $q=6$ & $q=7$ & $q=8$ & $q=9$ \\
		\midrule
		2 & 0.5711 & 0.4393 & 0.2855 & 0.2786 & 0.1757 & 0.2016 & 0.1428 & 0.1464 \\
		3 & 0.6672 & 0.6207 & 0.3336 & 0.3491 & 0.1862 & 0.3099 & 0.1668 & 0.2069 \\
		4 & 0.8001 & 0.9427 & 0.4000 & 0.4634 & 0.1985 & 0.5221 & 0.2000 & 0.3142 \\
		5 & 0.9873 & 1.5328 & 0.4936 & 1.2600 & 0.2129 & 0.9614 & 0.2468 & 0.5109 \\
		6 & 1.2564 & 2.6446 & 0.6282 & 2.4587 & 0.2300 & 1.9210 & 0.3141 & 0.8815 \\
		7 & 1.6511 & 4.7919 & 0.8256 & 5.1366 & 0.2503 & 4.1239 & 0.4128 & 1.5973 \\
		8 & 2.2414 & 9.0334 & 1.1207 & 11.3711 & 0.2748 & 9.4179 & 0.5604 & 3.0111 \\
		\bottomrule
	\end{tabular}
\end{table}

\section*{Acknowledgments}
The author wishes to thank the anonymous reviewers for their careful reading of the manuscript and for their valuable comments and suggestions.

\section{Declaration of no conflict of interest}

The author has no relevant financial or non-financial interests to disclose.
The author has no competing interests to declare that are relevant to the content of this article.
The author certifies that he has no affiliations with or involvement in any organization or entity with any financial interest or non-financial interest in the subject matter or materials discussed in this manuscript.
The author has no financial or proprietary interests in any material discussed in this article.


\begin{thebibliography}{9}
	
	\bibitem{Tom A}
	T.~M. Apostol,
	\textit{Introduction to Analytic Number Theory},
	Undergraduate Texts in Mathematics,
	Springer-Verlag, New York, 1976.
	
	\bibitem{Ba Gr}
	P.~T. Bateman and E.~Grosswald,
	\textit{On a theorem of Erdős and Szekeres},
	Illinois J. Math. \textbf{2} (1958), 88--98.
	
	\bibitem{Derbal 2023}
	O.~Bouakkaz and A.~Derbal,
	\textit{The mean value of the function $\frac{d(n)}{d^*(n)}$ in arithmetic progressions},
	Notes Number Theory Discrete Math. \textbf{29} (2023), no.~3, 445--453.
	
		\bibitem{Bord O} 
	O.~Bordellès, \textit{Arithmetic Tales. Advanced Edition}, Universitext, Springer, Cham, 2020.
	
	\bibitem{Karras 2016}
	A.~Derbal and M.~Karras,
	\textit{Valeurs moyennes d'une fonction liée aux diviseurs d'un nombre entier},
	C. R. Math. Acad. Sci. Paris \textbf{354} (2016), no.~5, 555--558.
	
   \bibitem{Karras 2018}
   M.~Karras and A.~Derbal,
   \textit{Mean value of an arithmetic function associated with the Piltz divisor function},
   Asian-Eur. J. Math. \textbf{13} (2020), no.~03, 2050062.
	


	
\end{thebibliography}
\end{document}